\documentclass{amsart}

\usepackage{amsfonts}
\usepackage{amssymb}
\usepackage{eucal}
\usepackage{graphicx}
\usepackage[all,cmtip]{xy}
\theoremstyle{plain}
\newtheorem{theorem}{\sc Theorem}[section]
\newtheorem{lemma}[theorem]{\sc Lemma}
\newtheorem{remark}[theorem]{\sc Remark}
\newtheorem{corollary}[theorem]{\sc Corollary}
\newtheorem{proposition}[theorem]{\sc Proposition}

\theoremstyle{definition}

\newtheorem*{example}{\sc Example}

\theoremstyle{remark}

\begin{document}
\title{On Topological Homotopy Groups of $n$-Hawaiian like spaces}
\author{H. Ghane}
\author{Z. Hamed}
\author{B. Mashayekhy}
\author{H. Mirebrahimi}
\address{Department of Pure Mathematics, Center of Excellence in Analysis on Algebraic Structures,
 Ferdowsi University of Mashhad, P.O.Box 1159-91775, Mashhad, Iran.}
\maketitle
\begin{abstract}
By an $n$-Hawaiian like space $X$ we mean the natural inverse limit,
$\displaystyle{\varprojlim (Y_i^{(n)},y_i^*)}$, where
$(Y_i^{(n)},y_i^*)=\bigvee_{j\leq i}(X_j^{(n)},x_j^*)$ is the wedge
of $X_j^{(n)}$'s in which $X_j^{(n)}$'s are $(n-1)$-connected,
 locally $(n-1)$-connected,
$n$-semilocally simply connected and compact CW spaces. In this paper, first we show that the
natural homomorphism $\displaystyle{\beta_n:\pi_n(X,*)\rightarrow
\varprojlim \pi_n(Y_i^{(n)},y_i^*)}$ is bijection. Second, using this fact we prove that
the topological $n$-homotopy group of an $n$-Hawaiian like space, $\pi_n^{top}(X,x^*)$, is a
topological group for all $n\geq 2$ which is a partial answer to the open question whether
$\pi_n^{top}(X,x^*)$ is a topological group for any space $X$ and $n\geq 1$. Moreover,
we show that $\pi_n^{top}(X,x^*)$ is metrizable.
\end{abstract}
%===============================================================================================
\section{Introduction}
In 2002, a work of Biss [1] initiated the development of a theory in
which the familiar fundamental group $\pi_1(X,x^*)$ of a topological
space $X$ becomes a topological space denoted by
$\pi_1^{top}(X,x^*)$ by endowing quotient topology inherited from
the path components of based loops in $X$ with compact-open
topology. An important feature of the theory is that if $X$ and $Y$
have the same homotopy type, then $\pi_1^{top}(X,x^*)$ and
$\pi_1^{top}(Y,y^*)$ are homeomorphic. Among the other things, he
claimed that $\pi_1^{top}(X,x^*)$ is a topological group and
$\pi_1^{top}$ is a functor from the category of based spaces to the
category of topological groups. However, there is a gap  in the
proof of $[1,Prop. 3.1]$, for more details see $[5]$ and $[2]$.

The authors [7] extended the above theory to higher homotopy groups
by introducing a topology on $n$-th homotopy group of a pointed
space $(X,x^*)$ as a quotient of the $n$-loop space $\Omega^n(X,x^*)$
equipped with the compact-open topology.  Call this space the topological homotopy
group and denote it by $\pi_n^{top}(X,x^*)$. As mentioned in [2], the
misstep in the proof is repeated by the authors to prove that
$\pi_n(X,x^*)$ is a topological group [7, Th. 2.1]. Hence, there is
a question  whether or not $\pi_n^{top}(X,x^*)$, $n\geq 1$, is a
topological group.

Note that if $X$ is locally contractible, then $\pi_1^{top}(X,x^*)$
inherits the discrete topology [5] and thus there is no information
other than algebraic data. The same thing happens in the case of higher homotopy
groups, when $X$ is a locally $n$-connected metric space, see [7,
Th. 3.6]. So spaces that are not locally $n$-connected, $n\geq 1$,
are interesting. One of the simplest nonlocally $n$-connected space
is the $n$-dimensional Hawaiian earring $\mathcal{H}_n$, $n\geq 1$.

Morgan and Morrison [10], among presenting a Van-Kampen theorem for
$1$-Hawaiian like spaces proved that the natural homomorphism
$$\displaystyle{\beta:\pi_1(\mathcal{H}_1,*)\rightarrow \varprojlim
\pi_1(Y_i^{(1)},y_i^*)}$$ is injective, where
$Y_i^{(1)}=\bigvee_{j\leq i}S_j^{1}$ is the wedge of $1$-spheres
$S_j^{1}$ of radius $\frac{1}{j}$ and center $(\frac{1}{j},0)$ and
hence $\displaystyle{\mathcal{H}_1=\varprojlim Y_i^{(1)}}$ with
respect to the natural inverse system for the $Y_i^{(1)}$.

Now consider the $n$-dimensional Hawaiian earring $\mathcal{H}_n$,
$n\geq 2$, which is the union of a sequence of $n$-spheres $S_j^{n}$
of radius $\frac{1}{j}$ identified at a common point $*$ as a
subspace of $\mathbb{R}^{n+1}$. This article aims to explore
$\mathcal{H}_n$ in the content of inverse limit space, i.e.
$\displaystyle{\mathcal{H}_n=\varprojlim Y_i^{(n)}}$, where
$\mathcal{H}_n$ can be approximated by factors
$Y_i^{(n)}=\bigvee_{j\leq i}S_j^{n}$.

As in [10], we consider a natural homomorphism
$$\displaystyle{\beta_n:\pi_n(\mathcal{H}_n,*)\rightarrow
\varprojlim \pi_n(Y_i^{(n)},y_i^*)}$$ as follows: let
$R_j:\mathcal{H}_n\rightarrow Y_j^{(n)}$ denote the projection
fixing $Y_j^{(n)}$ pointwise and collapsing
$\bigcup_{i=j+1}^{\infty}Y_i^{(n)}$ to the based point $*$. The
formula $\beta_n([f])=([R_1(f)],[R_2(f)],...)$ determines
the induced homomorphism $\beta_n$ into the inverse limit. In this paper, we intend to show
that $\beta_n$ is a bijection, for all $n\geq 2$. Then we present
two natural ways of imparting a topology on
$\pi_n(\mathcal{H}_n,*)$, for $n\geq 2$, as follows:\\

\begin{enumerate}
    \item {\it Since $\beta_n$ is an isomorphism,  one can pull back via
    $\beta_n$ to create the prodiscrete metric space
    $\pi_n^{lim}(\mathcal{H}_n,*)$. Indeed, as mentioned before,
    $Y_i^{(n)}$ is locally $n$-connected and thus by [7, Th. 3.6],
    $\pi_n^{top}(Y_i^{(n)},y_i^*)$ is discrete, which implies that
    $\displaystyle{ \varprojlim \pi_n^{top}(Y_i^{(n)},y_i^*)}$ is a
    prodiscrete metric space.}
    \item {\it We can endow the quotient topology on $\pi_n(\mathcal{H}_n,*)$
    inherited by the compact-open topology of $n$-loop space
    $\Omega^n(X,x^*)$, denoted by $\pi_n^{top}(\mathcal{H}_n,*)$, see
    [7].}
\end{enumerate}

In this paper, we will show that these two topologies
are agree. Therefore, $\pi_n^{top}(\mathcal{H}_n,*)$ is metrizable,
for $n\geq 2$. However, a result of Fabel [4] shows that the
topological fundamental group $\pi_1^{top}(\mathcal{H}_n,*)$ fails
to be metrizable. Moreover, we assert that
$\pi_n^{top}(\mathcal{H}_n,*)$ is a topological group, for $n\geq
2$. This statement answers the question whether $\pi_n^{top}(X,x^*)$
is a topological group, in special case. In fact, the main results
of the
paper are formulated as follows.
\begin{theorem} Suppose that for each $j$, $X_j^{(n)}$is an $(n-1)$-connected,
 locally $(n-1)$-connected,
$n$-semilocally simply connected, compact CW space and $X$ is
approximated by the factors $Y_i^{(n)}$, i.e.
$\displaystyle{(X,*)=\varprojlim (Y_i^{(n)},y_i^*)}$, where
$(Y_i^{(n)},y_i^*)=\bigvee_{j\leq i}(X_j^{(n)},x_j^*)$ is the wedge
of $X_j^{(n)}$'s. Then the homomorphism
$$\displaystyle{\beta_n:\pi_n(X,*)\rightarrow \varprojlim
\pi_n(Y_i^{(n)},y_i^*)}$$ is a bijection.
\end{theorem}

We call a space $X$ that satisfies the assumptions of Theorem 1.1 an
\emph{$n$-Hawaiian like} space.\\

\begin{theorem} If $X$ is an $n$-Hawaiian like space, then
$\pi_n^{top}(X,x^*)$ is a topological group, for $n\geq 2$.
Moreover, it is a prodiscrete metric space.
\end{theorem}

\section{Proof of Theorem 1.1}
First, we extend the Van-Kampen theorem [10] to higher homotopy
groups of $n$-Hawaiian like spaces. As a suitable model, one can
consider the $n$-dimensional Hawaiian earring $\mathcal{H}_n$.
 Eda and Kawamura [3] determined the $n$th homotopy
group of $\mathcal{H}_n$ by showing that $\pi_n(\mathcal{H}_n,*)$
is isomorphic to $\mathbb{Z}^\omega$. But, this section aims
 to determine $\pi_n(\mathcal{H}_n,*)$ by a form
of Van-Kampen theorem applicable to Hawaiian like spaces which is
needed to prove the further results.

We begin by fixing some notation. Let $X_j^{(n)}$, $j\in
\mathbb{N}$, be a based compact CW-complex which is also
$(n-1)$-connected,  locally $(n-1)$-connected, and $n$-semilocally
simply connected. Take the wedge $Y_i^{(n)}=\bigvee_{j\leq
i}X_j^{(n)}$ with collapsing maps $r_k^i:Y_i^{(n)}\rightarrow
Y_k^{(n)}$, where $k\leq i$, which are the identity on $Y_k^{(n)}$
and collapse $X_j^{(n)}$ to the base point $x_k^*$ if $k<j\leq i$.
Let $G_i=\bigoplus_{j=1}^i\pi_n(X_j^{(n)},x_j^*)$ with projections
$\pi_k^i:G_i\rightarrow G_k$ defined by
$\pi_k^i(\eta_1,\cdots,\eta_i)=(\eta_1,\cdots,\eta_k)$, $k\leq i$.
Then $\{Y_i^{(n)},r_k^i\}$ and $\{G_i,\pi_k^i\}$ are inverse systems
of topological spaces and groups whose limits we denote by
$\mathcal{H}_n$, and $\mathcal{G}$,
respectively.\\

\begin{example}
Consider $\mathcal{H}_n$ to be a finite family $\{\mathcal{H}_n^k\}_{k=1}^m$ of $n$-Hawaiian earring spaces,
which are joint to $m$ points of an $n$-sphere $S^n$, at their based points. One can
see that $\mathcal{H}_n$ is an $n$-Hawaiian like space. Indeed if $\mathcal{H}_n^k$
is approximated by $Y_{i,k}^{(n)}$'s, then $\displaystyle{\mathcal{H}_n=\varprojlim Y_i^{(n)}}$, where
$Y_i^{(n)}=(\cup_{k=1}^m Y_{i,k}^{(n)})\cup S^n$.\\
\end{example}

We start with a lemma from [6].
\begin{lemma}
Suppose that $X$ is an $(n-1)$-connected,  locally $(n-1)$-connected,
compact metric space and $\pi_n(X)$ is not finitely generated. Then
there exists $x\in X$ such that for each positive integer m, there
exists an $n$-loop $f_m$ at $x$ with diameter less than $2^{-m}$
which is not nullhomotopic. In particular, $X$ is not
$n$-semilocally simply connected at $x$.
\end{lemma}
The following assertion follows immediately.
\begin{corollary}
Let $X_j^{(n)}$, $j\in \mathbb{N}$, be as above. Then
$\pi_n(X_j^{(n)},x_j^*)$ is finitely generated.
\end{corollary}
For each $j\in \mathbb{N}$, we denote the generators of
$\pi_n(X_j^{(n)},x_j^*)$ by $\alpha_{j,1}, \cdots, \alpha_{j,k_j}$.

Now we recall a result of [12, Prop. 6.36.].

\begin{proposition}
If $X$ is an $n$-connected CW-complex and $Y$ is an $m$-connected
CW-complex, then the maps $i_X:(X,x^*)\rightarrow (X\vee Y,*)$
and $i_Y:(Y,y^*)\rightarrow (X\vee Y,*)$ given by
$i_X(x)=(x,y^*)$ and $i_Y(y)=(x^*,y)$ induce an isomorphism
$(i_{X_*},i_{Y_*}):\pi_k(X,x^*)\bigoplus \pi_k(Y,y^*)\rightarrow
\pi_k(X\vee Y,*)$ for $2\leq k\leq n+m$, provided $X$ or $Y$ is
locally finite.
\end{proposition}

So, one can determine $\pi_n(Y_i^{(n)},y_i^*)$ as follow (see [12]):

\begin{corollary}
With the previous notation, let $Y_i^{(n)}$ be the wedge
$\bigvee_{j\leq i}X_j^{(n)}$. Then
$$\pi_n(Y_i^{(n)},y_i^*)\cong\bigoplus_{j=1}^i\pi_n(X_j^{(n)},x_j^*)\ \ \ \ (for\ all\ n\geq 2).$$
\end{corollary}
\begin{remark}
By Corollaries $2.2$ and $2.4$, $\pi_n(Y_i^{(n)},y_i^*)$ is finitely
generated. Since $r_k^i(Y_i^{(n)})$ $=Y_k^{(n)}$, then clearly
$\alpha_{1,1},\cdots
,\alpha_{1,k_1},\alpha_{2,1},\cdots,\alpha_{2,k_2},\cdots,\alpha_{i,1},\cdots,\alpha_{i,k_i}$
are the generators of $\pi_n(Y_i^{(n)},y_i^*)$. This means that if
$\gamma\in\pi_n(Y_i^{(n)},y_i^*)$, then
$\gamma=(\alpha_{1,1}^{l_{1,1}}\cdots\alpha_{1,k_1}^{l_{1,k_1}},$
$\alpha_{2,1}^{l_{2,1}}\cdots
\alpha_{2,k_2}^{l_{2,k_2}},\cdots,\alpha_{i,1}^{l_{i,1}}\cdots\alpha_{i,k_i}^{l_{i,k_i}})$
for some integers
$l_{1,1},\cdots,l_{1,k_1},\cdots,$ $l_{i,1},\cdots,l_{i,k_i}$, where
$\alpha^l=\alpha\cdots\alpha$ is the concatenation $l$-times of
homotopy class of $\alpha$ with itself. We embed the generators
$\alpha_{j1},\cdots,\alpha_{jk_j}$, $(j\leq i)$, in
$\pi_n(Y_i^{(n)},y_i^*)$ by a map induced by inclusion
$X_j^{(n)}\rightarrow Y_i^{(n)}=\bigvee_{k=1}^iX_k^{(n)}$. For
simplicity, one can denote the embedded classes by the same
notations $\alpha_{j,1}\cdots \alpha_{j,k_j}$. It is easy to see
that the map
$$\alpha_{1,1}^{l_{1,1}}\cdots\alpha_{1,k_1}^{l_{1,k_1}}\cdots
 \alpha_{i,1}^{l_{i,1}}\cdots\alpha_{i,k_i}^{l_{i,k_i}} \mapsto
(\alpha_{1,1}^{l_{1,1}}\cdots\alpha_{1,k_1}^{l_{1,k_1}},\cdots,\alpha_{i,1}^{l_{i,1}}\cdots\alpha_{i,k_i}^{l_{i,k_i}})$$
induces an isomorphism between $\pi_n(Y_i^{(n)},y_i^*)$ and
$\bigoplus_{j=1}^i\pi_n(X_j^{(n)},x_j^*)$.
\end{remark}

Let $f:I^n\rightarrow \mathcal{H}_n$ be an $n$-loop
 with $f(\partial I^n)=\{*\}$, $f$ is said to be standard if $f(J_i^n)\subseteq
 X_i^{(n)}$, where $J_i^{n}=[a_i,b_i]\times I^{(n-1)}$ with
 $a_i=1-\frac{1}{2^{i-1}}$ and $b_i=1-\frac{1}{2^{i}}$, for $i\in
 \mathbb{N}$.

 Now, we recall a definition of [11]. Suppose that $(X,x)$ is a
 pointed space. Given an $n$-loop $f$ based at $x$ in $X$, then any
 other $n$-loop $g$ based at $x$ in $X$, with $H:g\simeq f (rel \partial
 I^n)$ and $g(I^n\setminus I_1^n)=\{x\}$ is called a concentration of $f$ on
 subcube $I_1^n$.

 We will need the following lemma which is a key step in the proof of
 Theorem 1.1.

\begin{lemma}
Each $n$-homotopy class in $\pi_n(\mathcal{H}_n,*)$ is represented by a
standard $n$-loop.
\end{lemma}
\begin{proof}
Let $f$ be any $n$-loop in $\mathcal{H}_n$ based at $*$. Then $f$
determines a sequence of $n$-loops $f_i$ in $Y_i^{(n)}$ defined by
$f_i=R_i\circ f$, where $R_i:\mathcal{H}_n\rightarrow Y_i^{(n)}$
denotes the retraction fixing $Y_i^{(n)}$ pointwise and collapsing
$\bigcup_{j=n+1}^{\infty}Y_j^{(n)}$ to the point $*$. By Corollary
2.4, the $n$-homotopy class of $f_i$ is contained in
$\bigoplus_{j=1}^i\pi_n(X_j^{(n)},x_j^*)$. Therefore,
$[f_i]=(\alpha_{1,1}^{l_{1,1}}\cdots\alpha_{1,k_1}^{l_{1,k_1}},\cdots,
\alpha_{i,1}^{l_{i,1}}\cdots\alpha_{i,k_i}^{l_{i,k_i}})$ for some
integers
$l_{1,1},\cdots,l_{1,k_1},\cdots,l_{i,1},\cdots,l_{i,k_i}$. Let
$g_1,\cdots,g_i$ be $n$-loops in $X_1^{(n)},\cdots,X_i^{(n)}$
representing the $n$-homotopy classes $\gamma_1,\cdots,\gamma_i$, where
$\gamma_j=\alpha_{j,1}^{l_{j,1}}\cdots\alpha_{j,k_j}^{l_{j,k_j}}$.
By Rremark $2.5$, $f_i\simeq g_1*\cdots *g_i$, where $*$ denotes the product
of $n$-loops in $\Omega^n(Y_i^{(n)},y_i^*)$. Note that each
$n$-loop in $X_j^{(n)}$ ($j<i$) can be embedded in $Y_i^{(n)}$ or in
 $\mathcal{H}_n$, if it is necessary, by
 maps induced by the inclusions $X_j^{(n)}\rightarrow Y_i^{(n)}$
 and $Y_i^{(n)}\rightarrow \mathcal{H}_n$, respectively.

Let $h_j$ be a concentration of $g_j$ on subcube $J_j^n$, for $j\in
\mathbb{N}$. By Lemma $2.5.2$ of [11], such concentration exists.
Then $f_i\simeq h_1*\cdots *h_i$ $(rel \partial I^n)$ by a homotopy
$H_i$.

We proceed by induction, constructing homotopies $H_i:I^n\times
[0,1]\rightarrow Y_i^{(n)}$ satisfying

\begin{enumerate}
    \item $H_i(x,0)=f_i(x)$;
    \item $H_i(x,1)=h_1*\cdots *h_i(x)$;
    \item $r_k^i\circ H_i=H_k$, $(k \leq i)$.
\end{enumerate}

Such homotopies give, in the limit, a homotopy $H$ of $n$-loop $f$
to a standard $n$-loop $h$ (the homotopies $H_i$'s are endowed with
the uniform metric, so the limit $H$ exists and it is continuous,
see [9, Theorem 46.8 and Corollary 46.6]).
\end{proof}

Now to prove the main result of this section, Theorem 1.1, suppose
$f$ is a standard $n$-loop based at $*$ in $\mathcal{H}_n$, $n\geq
2$, with corresponding sequence of $n$-loops $f_i=R_i\circ f$ in
$Y_i^{(n)}$. The homomorphism $\beta_n$ is well-defined, since
$r_k^i(f_i)=f_k$ and by construction of standard map in the above
lemma and Remark $2.5$, we have $\pi_k^i[f_i]=[f_k]$.

To show the injectivity of $\beta_n$, we must prove that given a
standard $n$-loop $f$ in $\mathcal{H}_n$ with
$\beta_n([f])=e_{\mathcal{G}}$, there is a based homotopy between
$f$ and the constant $n$-loop at $*$, where $e_{\mathcal{G}}$ is the
identity of $\displaystyle{\mathcal{G}=\varprojlim
\Pi_n(Y_i^{(n)},y_i^*) }$. Let $f_i=R_i\circ f$. Clearly
$[f_i]=(e_1,\cdots,e_i)$, where $e_j$ is the identity of
$\pi_n(X_j^{(n)},x_j^*)$, for $j=1,\cdots,i$. Then there are based
homotopies $K_i$ between $f_i$ and the constant $n$-loop at $y_i^*$.
Now the limit of $K_i$'s is a homotopy between $f$ and constant
$n$-loop at $*$, denoted by $K$.

Now, it is sufficient to show that $\beta_n$ is surjective. Let
$\mathfrak{g}=(g_i)\in \mathcal{G}$. Then
$g_i=(\eta_1,\cdots,\eta_i)\in
\bigoplus_{j=1}^i\pi_n(X_j^{(n)},x_j^*)\cong\pi_n(Y_i^{(n)},y_i^*)$.
Suppose $f_i$ represents $n$-homotopy class $g_i$. So
$r_j^i(f_i)=f_j$, since
$\pi_j^i(\eta_1,\cdots,\eta_i)=(\eta_1,\cdots,\eta_j)$ for $j\leq
i$. Also there are integers
$l_{1,1},\cdots,l_{1,k_1},\cdots,l_{i,1},\cdots,l_{i,k_i}$ such that
$[f_i]=\alpha_{1,1}^{l_{1,1}}\cdots\alpha_{1,k_1}^{l_{1,k_1}}\cdots
$ $\alpha_{i,1}^{l_{i,1}}\cdots\alpha_{i,k_i}^{l_{i,k_i}}$. Let the
$n$-loop $t_j$ represent
$\alpha_{j,1}^{l_{j,1}}\cdots\alpha_{j,k_j}^{l_{j,k_j}}$, let $s_j$
be a concentration of $t_j$ on $J_j^n$, and let $h_i=s_1\cdots s_i$.
Clearly the limit of $h_i$'s is an $n$-loop denoted by $h$ and
$[h]=g$ (the homotopies $K_i$'s and the $n$-loops $h_i$'s are
endowed with the uniform metric so their limits $K$ and $h$ exist
and they are also continuous, see [M, Theorem 46.8, Corollary
46.6]). This completes the proof.

\begin{corollary}
Let $\mathcal{H}_n$ be the $n$-dimensional Hawaiian earring with
based point *. Then $\displaystyle{\pi_n(\mathcal{H}_n,*)\cong
\mathcal{G}=\varprojlim G_i}$, where $G_i$ is the direct sum of
$i$ copies of integers $\mathbb{Z}$.
\end{corollary}

%===============================================================================================

\section{Proof of Theorem 1.2}

In conclusion, we assert that $\pi_n^{top}(\mathcal{H}_n,*)$ is a topological
group homeomorphic to $\pi_n^{lim}(\mathcal{H}_n,*)$ which implies that $\pi_n^{top}(\mathcal{H}_n,*)$
is metrizable.

As mentioned in introduction, it is an open question whether or not in general
$\pi_n^{top}(X,*)$ is a topological group. If $X$ is a locally
$n$-connected metric space, then $\pi_n^{top}(X,x^*)$ and hence
$\pi_n^{top}(X,x^*) \times \pi_n^{top}(X,x^*)$ is discrete (see [7])
and therefore multiplication is continuous. In general, the
continuity of multiplication remains an unsettled question.

The following lemma shows that if $(X,x^*)$ is a pointed topological space, then left and right translations
by a fixed element in $\pi_n^{top}(X,x^*)$ are homeomorphisms.

%----------------------------------------------------------------------------------------------
\begin{lemma}
Let $(X,x^*)$ be a pointed topological space. If $[f] \in \pi_n^{top}(X,x^*)$, then left
and right translations by $[f]$  are homeomorphisms of $\pi_n^{top}(X,x^*)$.
\end{lemma}
\begin{proof}
First, we show that the multiplication
$$\Omega^n(X,x^*)\times \Omega^n(X,x^*) \stackrel{\tilde{m}}{\longrightarrow} \Omega^n(X,x^*)$$
is continuous, where $\tilde{m}$ is concatenation of $n$-loops and $\Omega^n(X,x^*)$ is equipped
with compact-open topology.

Let $\langle K,U\rangle$ be a subbasis
element in $\Omega^n(X,x^*)$. Define
$$K_1=\{(t_1,\ldots,t_n);(t_1,\ldots,t_{n-1},\frac{t_n}{2})\in K\}$$
and $$K_2=\{(t_1,\ldots,t_n);(t_1,\ldots,t_{n-1},\frac{t_n+1}{2})\in
K\}.$$ Then
\[\tilde{m}^{-1}(\langle K,U\rangle)=\{(f_1,f_2);(f_1*f_2)(K)\subseteq U\}=\langle K_1,U\rangle\times\langle
K_2,U\rangle\]
is open in $\Omega^n(X,x^*)\times
\Omega^n(X,x^*)$ and so $\tilde{m}$ is continuous.

Now, fix $[f] \in \pi_n^{top}(X,x^*)$ and consider left translation by $[f]$ on $\pi_n^{top}(X,x^*)$
$$\pi_n^{top}(X,x^*)\rightarrow\pi_n^{top}(X,x^*).$$
$$\ \ \ [g]\mapsto[f]\cdot[g]$$
Clearly, the following diagram is commutative

\[
    \xymatrix{
        \Omega^n(X,x^\ast)\ar[r]^{\tilde{m}_f}\ar[d] & \Omega^n(X,x^\ast)\ar[d]\\
        \pi_n^{top}(X,x^\ast)\ar[r]^{m_{[f]}} &
        \pi_n^{top}(X,x^\ast),
    }
\]
where $\tilde{m}_f$ is defined by $g\mapsto f \ast g$. By the
universal property of quotient maps [9,Theorem 11.1], $m_{[f]}$ is
continuous. Since $m_{[f]}^{-1} = m_{[f^{-1}]}$ is also continuous,
so $m_{[f]}$ is a homeomorphism, as desired. A similar argument
implies that right translation is also a homeomorphism.
\end{proof}

Note that $\pi_n^{top}(X,x^*)$ acts on itself by left and right translations as a group of homeomorphisms.
It is easy to see that these actions are both transitive. So, we have the following result.

\begin{proposition}
If $(X,x^*)$ is a pointed topological space, then $\pi_n^{top}(X,x^*)$ is a homogeneous space.
\end{proposition}

Now, let $U$ be an open neighborhood of $x^*$ in $X$ and $\tilde{U}$ be the set of all $n$-loops
based at $x^*$ lying inside $U$. Also let $\hat{U}$ be its quotient under homotopy, that is
$$\hat{U}=\{ [f]; f\in \Omega^n(X,x^*)\text{ and } im f\subseteq U\}.$$
Suppose $W\subset \pi_n^{top}(X,x^*)$ is an open neighborhood containing the identity element $[e_{x_*}]$.
Then
$$e_{x_*}\in q^{-1}(W)=\bigcup_{\alpha \in J}(\bigcap_{i=1}^{k_\alpha}<K_i^\alpha , V_i^\alpha>),$$
where $e_{x_*}$ is the constant n-loop based at $x^*$. Therefore, there exists an index $\alpha \in J$
such that $e_{x_*} \in \bigcap_{i=1}^{k_\alpha}<K_i^\alpha , V_i^\alpha>$, which implies that
$$x^* \in \bigcup_{i=1}^{k_\alpha}V_i^\alpha =V^\alpha \ ,\ \
\ \tilde{V}^{\alpha}\subset\bigcap_{i=1}^{k_\alpha}<K_i^\alpha , V_i^\alpha> \ \subset \ q^{-1}(W)$$
and then $[e_{x_*}]\in \hat{V}^{\alpha}\subset W$.

Since left translation is continuous in $\pi_n^{top}(X,x^*)$,
$$[f]W=\{[f]\cdot [g];\ [g]\in W\}$$ runs through a basis at $[f]$ for $\pi_n^{top}(X,x^*)$ as $W$
runs through a basis at $[e_{x_*}]$ in its topology.

Now, we use a classical theorem in the theory of topological groups [8] which
asserts that for given a group $G$ with a filter base $\{U\}$
satisfying the following conditions

\begin{itemize}

    \item  Each $U$ is symmetric, i.e. $U^{-1}=U$;

    \item For each $U$ in $\{U\}$, there exists a $V$ in $\{U\}$ such that $V^2\subset U$, where $V^2=\{xy; x,y\in V\}$;

    \item For each $U$ in $\{U\}$ and $a\in G$, there exists a $V$ in $\{U\}$ such that $V\subset a^{-1}Ua$ or $aVa^{-1}\subset
    U$,

\end{itemize}
then $\{U\}$ forms a fundamental system of neighborhoods of $e$. In
particular, $G$ with the topology induced by this fundamental system
becomes a topological group.

Since $\pi_n(X,x^*)$ is an abelian group, for $n\geq 2$,  it is easy to see that
the filter base $\{\hat{U}\}$
forms a fundamental system of neighborhoods of the identity element  $e$ and hence $\pi_n(X,x^*)$ with this topology
becomes a topological group, denoted by $\pi_n^{lim}(X,x^*)$. By the above statements, this
topology, denoted by $\tau^{lim}$, is coarser than quotient topology $\tau^{top}$ on
$\pi_n(X,x^*)$ inherited from $\Omega^n(X,x^*)$ with the compact-open topology.

If $X=\mathcal{H}_n$ and $W_i^{(n)}$ is an n-connected neighborhood of $x_i^*$ in $X_i^{(n)}$, then
the following sets provide a basis for $\mathcal{H}_n$ at $*$
$$\mathcal{U}_k=(\bigcup_{i\leq k}W_i^{(n)})\cup (\bigcup_{i>k}X_i^{(n)}).$$
Also, a basis of neighborhoods of the identity $e_{\mathcal{G}}$ in $\mathcal{G}$ is given by the subgroups
$$\mathcal{G}_k=\{\mathfrak{g}=(g_i)\in \mathcal{G};\ g_k=(e_1,\ldots,e_k)\},\ \ for\ k\in \mathbb{N}.$$

\begin{theorem}
The map $\beta_n :\pi_n^{lim}(\mathcal{H}_n,*)\rightarrow \mathcal{G}$ is a homeomorphism and
therefore an isomorphism of topological groups.
\end{theorem}

\begin{proof}
Since $\{\hat{\mathcal{U}}_k\}$ and $\{\mathcal{G}_k\}$ form bases
at $*$ and $e_{\mathcal{G}}$ for topologies on
$\pi_n^{lim}(\mathcal{H}_n,*)$ and $\mathcal{G}$ (see also the
statement before Theorem 3.3), it is sufficient to show that
$\beta_n(\hat{\mathcal{U}}_k)=\mathcal{G}_k$. For, let $[f] \in
\hat{\mathcal{U}}_k$ and $h$ be a standard $n$-loop representing
$[f]$. Since $W_i^{(n)}$ is
 $n$-connected for $i\leq k$, we have $h_k=R_k \circ h$ is nullhomotopic. Therefore,
$\beta_n ([f])\in \mathcal{G}_k$.

Conversely, if $\mathfrak{g}=(g_i) \in \mathcal{G}_k$, then
$g_k=(e_1, \ldots , e_k)$. The homomorphism $\beta_n$ is a
bijection, so there is a standard $n$-loop $f$ such that $[f] \in
\pi_n(\mathcal{H}_n,*)$ and $\beta_n ([f])=g$. But $\beta_n
([f])=([R_1(f)],[R_2(f)], \ldots)$. This implies that
$[R_k(f)]=(e_1, \ldots , e_k)$ and therefore, $R_k(f)=f_k$ is
nullhomotopic.
 Take $h$ a standard
 $n$-loop such that $h\simeq f \ (rel\ \partial I^n)$ and $R_k\circ h=y_k^*$.
Then $\beta _n\ [h]=\beta _n\ [f]\ \in \hat{\mathcal{U}}_k$.
\end{proof}
The homomorphism $\beta_n$ gives a compatible sequence of
homomorphisms $\beta_{n,i} :\pi_n(\mathcal{H}_n,*) \rightarrow
\pi_n(Y_i^{(n)},y_i^*)$. By [7], the homomorphisms $\beta_{n,i}
:\pi_n^{top}(\mathcal{H}_n,*) \rightarrow
\pi_n^{top}(Y_i^{(n)},y_i^*)$ are continuous when we dealing with
topological homotopy groups, implies that $\beta_n
:\pi_n^{top}(\mathcal{H}_n,*)\rightarrow
\mathcal{G}=\underleftarrow{lim}\ \pi_n^{top}(Y_i^{(n)},y_i^*)$ is
also continuous.

So that there is no ambiguity in notation, we denote $\beta_n :\pi_n^{lim}(\mathcal{H}_n,*)\rightarrow \mathcal{G}$
 and $\beta_n :\pi_n^{top}(\mathcal{H}_n,*)\rightarrow \mathcal{G}$ by $\beta_n^{lim}$ and $\beta_n^{top}$, respectively.\\

 Now, consider the following commutative diagram.

\[
    \xymatrix{
        \pi_n^{top}(\mathcal H_n,*)\ar[d]_{id}\ar[r]^-{\beta_n^{top}} & \mathcal G\\
        \pi_n^{lim}(\mathcal H_n,*)\ar[ur]_-{\beta_n^{lim}}
    }
\]
 Since $\beta_n^{lim}$ is a homeomorphism, the identity map $i$ is continuous. This fact shows that quotient topology $\tau^{top}$ inherited of compact-open topology of $n$-loop space is coarser than $\tau^{lim}$. But we have already seen that $\tau^{top} \subset \tau^{lim}$. Therefore, these two topologies on $\pi_n(\mathcal{H}_n,*)$ are equivalent. This means that $\pi_n^{top}(\mathcal{H}_n,*)$
 is a topological group and also a prodiscrete metric space.\\
%===============================================================================================

$\mathbf{Acknowledgement}$: Authors are grateful to the referee for
valuable suggestions and useful remarks.
%===============================================================================================

\end{document}